\documentclass[11pt,a4paper,showkeys,reqno]{amsart}
\usepackage{amssymb,amsthm,amsmath,amsxtra,here}
\usepackage{indentfirst}
\usepackage{hyperref}
\usepackage{fullpage}
\usepackage[all]{xy}

\usepackage{mathtools}
\mathtoolsset{showonlyrefs,showmanualtags}
\hypersetup{colorlinks=true,urlcolor=blue,citecolor=blue,linkcolor=blue}
\theoremstyle{plain}
\newtheorem{thm}{Theorem}%定理等の最初の数字もequationと同じものにする
\newtheorem{prop}[thm]{Proposition}
\newtheorem{lem}[thm]{Lemma}

\newtheorem{cor}[thm]{Corollary}
\numberwithin{thm}{section}
\theoremstyle{definition}

\newtheorem{rem}[equation]{Remark}
\newtheorem*{ac}{Acknowledgements}%「*」は定理番号を付けない
\DeclareMathOperator{\rank}{rank}
\DeclareMathOperator{\corank}{corank}
\newcommand{\bbC}{\mathbb C}
\newcommand{\bbH}{\mathbb H}

\newcommand{\bbQ}{\mathbb Q}
\newcommand{\bbR}{\mathbb R}
\newcommand{\bbZ}{\mathbb Z}

\newcommand{\calO}{\mathcal O}
\newcommand{\calG}{\mathcal G}

\newcommand{\frakf}{\mathfrak{f}}
\newcommand{\frakp}{\mathfrak{p}}
\newcommand{\fraka}{\mathfrak{a}}
\newcommand{\character}[2]{\left[\begin{array}{c}#1 \\#2\end{array}\right]}
\newcommand{\absolute}[1]{\left|#1 \right|}
\newcommand{\skakko}[1]{\left(#1 \right)}
\newcommand{\mkakko}[1]{\left\{#1 \right\}}

\usepackage[OT2,T1]{fontenc}
\DeclareSymbolFont{cyrletters}{OT2}{wncyr}{m}{n}
\DeclareMathSymbol{\Sha}{\mathalpha}{cyrletters}{"58}
\title{{\Large The rank of a CM elliptic curve and a recurrence formula}}
\author{Keiichiro Nomoto}
\address{FACULTY OF MATHEMATICS, KYUSHU UNIVERSITY, MOTOOKA 744, NISHI-KU FUKUOKA 819-0395, JAPAN}
\email{nomotokeiichiro@gmail.com}
\date{}
\keywords{elliptic curve, $L$-function, recurrence formula}
\subjclass[2010]{Primary~11F67, Secondary~11G40}
\begin{document}
\maketitle
\begin{abstract}
Let $p$ be a prime number and $E_{p}$ denote the elliptic curve $y^2=x^3+px$. It is known that for $p$ which is congruent to $1, 9$ modulo $16$, the rank of $E_{p}$ over $\bbQ$ is equal to $0, 2$. Under the condition that the Birch and Swinnerton-Dyer conjecture is true, we give a necessary and sufficient condition that the rank is $2$ in terms of the constant term of some polynomial that is defined by a recurrence formula.
\end{abstract}

\section{Introduction}

Which prime number $p$ can be written as the sum of two cubes of rational numbers? This is one of the classical Diophantine problems and there are various works. ({\it cf}. \cite{DasguptaVoight}, \cite{Yin}) This problem is equivalent to the existence of non-torsion $\bbQ$-rational points of the curve $A_p: x^3+y^3=p$. The curve $A_p$ has the structure of an elliptic curve over $\bbQ$ with the point $\infty=[1: -1: 0]$. For an odd prime number $p$, we have $A_p(\bbQ)_{\rm tors}=\{\infty\}$. Therefore an odd prime number $p$ is written as the sum of cubes if and only if the rank of $A_p$ over $\bbQ$ is not $0$. \cite{Satge} shows the upper bound
\begin{align}
\rank A_p(\bbQ)\leq \begin{cases}
0 & (p\equiv 2, 5\mod 9),\\
1 & (p\equiv 4, 7, 8\mod 9),\\
2 & (p\equiv 1\mod 9).
\end{cases}
\end{align}
Let $\varepsilon(A_p/\bbQ)$ be the sign of the functional equation for the Hasse-Weil $L$ function of $A_p$ over $\bbQ$. The parity conjecture for $3$-Selmer groups ({\it cf}. \cite{Nekovar}) leads to
\begin{align}
(-1)^{\corank {\rm Sel}_{3^{\infty}}(A_p/\bbQ)}=\varepsilon(A_p/\bbQ)=\begin{cases}
+1 & (p\equiv 1, 2, 5\mod 9),\\
-1 & ({\rm otherwise}).
\end{cases}
\end{align}
Thus for the case where $p\equiv 1\bmod 9$ (resp. $p\equiv 4, 7, 8\bmod 9$), the rank of $A_p$ is $0$ or $2$ (resp. $1$) if we assume the Tate-Shafarevich group is finite. The remaining problem is essentially whether the rank is $0$ or $2$ for the case where $p$ is congruent to $1$ modulo $9$. In the paper \cite{VillegasZagier}, Villegas and Zagier have given three necessary and sufficient conditions that the rank is equal to $2$ under the Birch and Swinnerton-Dyer(BSD) conjecture. One of the conditions is described in terms of a recurrence formula although they did not give the details of the proof.

In this paper, we give a similar formula for the elliptic curve $E_p: y^2=x^3+px$. A $2$-descent (\cite{Silverman}) shows the upper bound
\begin{align}
\rank E_{p}(\bbQ)\leq \begin{cases}
0 & (p\equiv 7, 11\mod 16),\\
1 & (p\equiv 3, 5, 13, 15\mod 16),\\
2 & (p\equiv 1, 9\mod 16).
\end{cases}
\end{align}
For the case where $p$ is congruent to $1, 9$ modulo $16$, the sign of functional equation of the Hasse-Weil $L$-function of $E_{p}$ over $\bbQ$ is $+1$. Similarly for the case $E_{p}$, we see that $\rank E_{p}(\bbQ)=0$ or $2$ if we assume the Tate-Shafarevich group is finite. Let $\Omega_{E}=\Gamma(1/4)^2/(2\varpi^{1/2})$ be the real period of $E_{1}$ and let $S_{p}$ be the constant satisfying
\begin{align}
L(E_{p}/\bbQ, 1)=\dfrac{\Omega_{E}S_{p}}{2p^{1/4}}.\label{eq:Sp}
\end{align}
Here $\varpi=3.1415\cdots$. The BSD conjecture predicts that the constant $S_{p}$ is equal to the order of the Tate-Shafarevich group if $\rank E_{p}(\bbQ)= 0$ and is $0$ otherwise.

\begin{thm}\label{thm:mainthm1}
Let $p$ be a prime number which is congruent to $1, 9$ modulo $16$. Suppose that $S_p\in \bbZ$. If the rank of $E_{p}$ over $\bbQ$ is equal to $2$, then $p$ divides $f_{3(p-1)/8}(0)$, where the polynomial $f_n(t)\in \mathbb{Z}[t]$ is defined by the recurrence formula
\begin{align}
f_{n+1}(t)=-12(t+1)(t+2)f_n'(t)+(4n+1)(2t+3)f_n(t)-2n(2n-1)(t^2+3t+3)f_{n-1}(t).
\end{align}
The initial condition is $f_0(t)=1,\,f_1(t)=2t+3$. Moreover if we assume the BSD conjecture, then the converse is also true. 
\end{thm}

% The constant $S_p$ appeared in Theorem \ref{thm:mainthm1} is conjectured to be the order of the Tate-Shafarevich group $\Sha(E_p/\bbQ)$ by BSD conjecture (see precise definition \eqref{eq:Sp}). 

We will show that $S_p$ is a rational integer for more general elliptic curves in the next paper.

We tried to recover the proof of \cite[Theorem 3]{VillegasZagier} which claims the same as Theorem \ref{thm:VZthm}. As a result, we obtain Theorem \ref{thm:mainthm2} below although we could not obtain the proof of Theorem \ref{thm:VZthm}. Our recurrence formula \eqref{eq:RFmainthm2} is simpler than \eqref{eq:VZformula}. In Table \ref{tab:VZformula} and Table \ref{tab:RFmainthm2}, we show the first several terms for the two recurrence formulas. The degree of the polynomial and the number of terms of \eqref{eq:RFmainthm2} are less than \eqref{eq:VZformula}. Perhaps we may make the recurrence formula \eqref{eq:RFmainthm2} simpler. A procedure to obtain the recurrence formula \eqref{eq:RFmainthm2} is essentially the same as \cite{VillegasZagier}.  

\begin{thm}[{\cite[Theorem 3]{VillegasZagier}}]\label{thm:VZthm}
Let $p$ be a prime number which is congruent to $1$ modulo $9$, the rank of $A_p$ over $\bbQ$ is equal to $2$, then $p$ divides $a_{(p-1)/3}(0)$, where
the polynomial $a_n(t)\in \mathbb{Z}[t]$ is defined by the recurrence formula
\begin{align}
a_{n+1}(t)=-(1-8t^3)a_n'(t)-(16n+3)t^2a_n(t)-4n(2n-1)ta_{n-1}(t).\label{eq:VZformula}
\end{align}
The initial condition is $a_0(t)=1,\,a_1(t)=-3t^2$. Moreover if we assume the BSD conjecture, then the converse is also true.
\end{thm}

\begin{thm}\label{thm:mainthm2}
Let $p$ be a prime number which is congruent to $1$ modulo $9$, the rank of $A_p$ over $\bbQ$ is equal to $2$, then $p$ divides $x_{(p-1)/3}(0)$, where the polynomial $x_n(t)\in \mathbb{Z}[t]$ is defined by the recurrence formula
\begin{align}
x_{n+1}(t)=-2(1-8t^3)x_n'(t)-8nt^2x_n(t)-n(2n-1)tx_{n-1}(t).\label{eq:RFmainthm2}
\end{align}
The initial condition is $x_0(t)=1,\,x_1(t)=0$. Moreover if we assume the BSD conjecture, then the converse is also true.
\end{thm}

We now discuss the proof of Theorem \ref{thm:mainthm1}. For the case where $p$ is congruent to $1, 9$ modulo $16$, we see that $\rank E_{p}(\bbQ)=2$ if and only if $L(E_{p}/\bbQ, 1)=0$ under the BSD conjecture. The calculation $L(E_{p}/\bbQ, 1)$ reduces to $L(\psi^{2k-1}, k)$ for some Hecke character $\psi$ and some positive integer $k$. More precisely, by a theory of $p$-adic $L$-functions, there exists a $\bmod~p$ congruence relation between an algebraic part of $L(E_{p}/\bbQ, 1)$ and that of $L(\psi^{2k-1}, k)$. Therefore with the estimate $|L(E_{p}/\bbQ, 1)|$,  it holds that $L(E_p/\bbQ, 1)=0$ if and only if $p$ divides the algebraic part of $L(\psi^{2k-1}, k)$. We write the algebraic part of $L(\psi^{2k-1}, k)$ in terms of a recurrence formula by using the method of \cite{Villegas}. 

In Section $2.1$, we show the rank of $E_{p}$ is equal to $2$ if and only if $p$ divides the algebraic part of $L(\psi^{2k-1}, k)$. In Section $2.2$, we review some basic properties of the Maass-Shimura operator $\partial_k$. In Section 2.3, we write the special value $L(\psi^{2k-1}, k)$ with some special value of $\partial_k$-derivative of some modular form. In Section 3, we write the special value of $\partial_k$-derivative of the modular form as the constant term of some polynomial that is defined by a recurrence formula.

\newpage

\begin{table}[H]\label{tab:VZformula}
\centering
\begin{tabular}{|c||l|}
\hline
$n$ & $a_n(t)$\\ \hline \hline
$0$ & $1$ \\ \hline
$1$ & $-3t^2$ \\ \hline
$2$ & $9t^4+2t$ \\ \hline
$3$ & $-27t^6-18t^3-2$ \\ \hline
$4$ & $81t^8+108t^5+36t^2$ \\ \hline
$5$ & $-243t^{10}-540t^7-360t^4+152t$ \\ \hline
$6$ & $729t^{12}+2430t^9+2700t^6-16440t^3-152$ \\ \hline
$7$ & $-2187t^{14}+10206t^{11}-17010t^8+1311840t^5+24240t^2$ \\ \hline
$8$ & $6561t^{16}+40824t^{13}+95256t^{10}-99234720t^7-2974800t^4+6848t$ \\ \hline
$9$ & $-19683t^{18}-157464t^{15}-489888t^{12}+7449816240t^9+359465040t^6-578304t^3-6848$ \\ \hline
\end{tabular}
\caption{the recurrence formula for $a_n(t)$}
\end{table}

\begin{table}[H]
\begin{tabular}{cc}
\begin{minipage}[h]{7cm}
\centering
\begin{tabular}{|c||l|}
\hline
$n$ & $x_n(t)$\\ \hline \hline
$0$ & $1$ \\ \hline
$1$ & $0$ \\ \hline
$2$ & $-t$ \\ \hline
$3$ & $2$ \\ \hline
$4$ & $-33t^2$ \\ \hline
$5$ & $76t$ \\ \hline
$6$ & $-339t^3$ \\ \hline
$7$ & $4314t^2$ \\ \hline
$8$ & $-72687t^4-3424t$ \\ \hline
$9$ & $228168t^3+6848$\\ \hline
\end{tabular}
\caption{the recurrence formula for $x_n(t)$}
\label{tab:RFmainthm2}
\end{minipage}
\begin{minipage}[h]{8cm}
\centering
\begin{tabular}{|c||l||c||l|}
\hline
$p$ & $p|f_{3(p-1)/8}(0)$ & $p$ & $p|f_{3(p-1)/8}(0)$ \\ \hline \hline
$17$ & false & 257 & false \\ \hline
$41$ & false & 281 & true \\ \hline
$73$ & true & 313 & false \\ \hline
$89$ & true & 337 & true \\ \hline
$97$ & false & 353 & true \\ \hline
$113$ & true & 401 & false \\ \hline
$137$ & false & 409 & false \\ \hline
$193$ & false & 433 & false \\ \hline
$233$ & true & 449 & false \\ \hline
$241$ & false & 457 & false \\ \hline
\end{tabular}
\caption{the constant term $f_{3(p-1)/8}(0)$}
\end{minipage}
\end{tabular}
\end{table}

\begin{table}[H]
\label{table:RFmainthm1}
\centering
\begin{tabular}{|c||l|}\hline
$n$ & $f_n(t)$\\ \hline \hline
$0$ & $1$ \\ \hline
$1$ & $2t+3$ \\ \hline
$2$ & $-6t^2-18t-9$ \\ \hline
$3$ & $12t^3+54t^2+108t+81$ \\ \hline
$4$ & $60t^4+360t^3+1296t^2+2268t+1377$ \\ \hline
$5$ & $-1512t^5-11340t^4-\dots-34992t^2-13122t+2187$ \\ \hline
$6$ & $21816t^6+196344t^5+\dots+1027890t^2+433026t+80919$ \\ \hline
$7$ & $-280368t^7-2943864t^6-\dots-46517490t^2-24074496t-5189751$ \\ \hline
$8$ & $3319056t^8+39828672t^7+\dots+1016482608t^2+423420696t+82097793$ \\ \hline
$9$ & $-32283360t^9-435825360t^8-\dots+2060573904t^2+4373050842t+1702205523$ \\ \hline
\end{tabular}
\caption{the recurrence formula for $f_n(t)$}
\end{table}

\section{The algebraic part of the special value $L(E_{p}/\bbQ, 1)$}

\subsection{congruence of a special value of $L$-function}

Here we show that there exists a mod $p$ congruence relation between the special value $L(E_p/\bbQ, 1)$ and some special value of a Hecke $L$-function associated to the elliptic curve $E_1: y^2=x^3+x$.\\

Suppose $p$ satisfies $p\equiv 1, 9\bmod 16$ and $p$ splits as $\frakp\bar{\frakp}$ in the integer ring of $K=\bbQ(i)$. If necessary by repalcing $\bar{\frakp}$ by $\frakp$,  we may assume there is a generator $\pi=a+bi$ of $\frakp$ satisfying
\begin{align}
a\equiv 1\mod 4, \ \ b\equiv -\skakko{\dfrac{p-1}{2}}!a\mod p.
\end{align}
We fix inclusions $i_\infty: \overline{\bbQ}\hookrightarrow \bbC, i_p: \overline{\bbQ}\hookrightarrow \bbC_p$ so that $i_p$ is compatible with $\frakp$-adic topology. Let $\Omega_{E}=\Gamma(1/4)^2/(2\varpi^{1/2})$ be the real period of $E_{1}$ and let $S_{p}$ be the constant satisfying
\begin{align}
L(E_{p}/\bbQ, 1)=\dfrac{\Omega_{E}S_{p}}{2p^{1/4}}.\label{eq:Sp}
\end{align}
The BSD conjecture predicts that the constant $S_{p}$ is equal to the order of the Tate-Shafarevich group if $\rank E_{p}(\bbQ)= 0$ and is $0$ otherwise. The elliptic curve $E_{1}: y^2=x^3+x$ has complex multiplication by $\calO_K$. Let $\psi$ be the Hecke character of $K$ associated to $E_{1}$ and let $\chi$ be the quartic character such that $L(E_p/\bbQ, s)=L(\psi\chi, s)$. These characters are explicitly given by
\begin{align}
&\psi(\fraka)=\skakko{\dfrac{-1}{\alpha}}_4\alpha=(-1)^{(a-1)/2}\alpha \ \ \ \text{if} \  (\fraka, 4)=1,\\ &\chi(\fraka)=\overline{\skakko{\dfrac{\alpha}{p}}_4} \ \ \  \text{if} \ (\fraka, p)=1,
\end{align}
where $\alpha=a+bi$ is the primary generator of $\fraka$ and $(\cdot/\cdot)_4$ is the quartic residue character ({\it cf}. \cite[II, Exercice 2.34]{Silverman}). Let $k$ be a positive interger. We define the algebraic part of $L(\psi^{2k-1}, k)$ to be
\begin{align}
L_{E, k}=\dfrac{2^{k+1}3^{k-1}\varpi^{k-1}(k-1)!}{\Omega_{E}^{2k-1}}L(\psi^{2k-1}, k).
\end{align}

\begin{lem}\label{lem:CRT}
Let $p$ be a prime number such that $p\equiv 1, 9\bmod 16$ and $k=(3p+1)/4$. For all non-zero integral ideals $\fraka$ of $\calO_K$ which is prime to $4p$, we have
\begin{align}
\chi(\fraka)\equiv \skakko{\dfrac{\alpha}{\overline{\alpha}}}^{k-1} \mod p,
\end{align}
where $\alpha$ is the primary generator of $\fraka$.
\end{lem}
\begin{proof}
Since $3(N(\pi)-1)=4(k-1)$, we have
\begin{align}
\alpha^{k-1}\equiv \skakko{\dfrac{\alpha^3}{\pi}}_4 \mod \pi, \ \alpha^{k-1}\equiv \skakko{\dfrac{\alpha^3}{\overline{\pi}}}_4 \mod \overline{\pi}.
\end{align}
We take $a\in \frakp, b\in \bar{\frakp}$ so that $a+b=1$. Then by the Chinese Remainder Theorem, we have
\begin{align}
&\alpha^{k-1}\equiv a\skakko{\dfrac{\alpha^3}{\overline{\pi}}}_4+b\skakko{\dfrac{\alpha^3}{\pi}}_4 \mod p\calO_K,\label{eq:CRT1}\\
&\overline{\alpha}^{k-1} \equiv a\skakko{\dfrac{\overline{\alpha}^3}{\overline{\pi}}}_4+b\skakko{\dfrac{\overline{\alpha}^3}{\pi}}_4 \mod p\calO_K\label{eq:CRT2}.
\end{align}
Since the equation \eqref{eq:CRT1} multiplied by $(\overline{\alpha}^3/\pi)_4$ equals to the equation \eqref{eq:CRT2} multiplied by $(\alpha^3/\pi)_4$, it holds that
\begin{align}
\skakko{\dfrac{\overline{\alpha}^3}{\pi}}_4\alpha^{k-1}\equiv \skakko{\dfrac{\alpha^3}{\pi}}_4\overline{\alpha}^{k-1} \mod p\calO_K.
\end{align}
Therefore we obtain
\begin{align}
\dfrac{\alpha^{k-1}}{\overline{\alpha}^{k-1}}\equiv \skakko{\dfrac{\alpha^3}{\pi}}_4\skakko{\dfrac{\alpha^3}{\overline{\pi}}}_4=\skakko{\dfrac{\alpha}{p}}^3=\chi(\fraka) \mod p\calO_K.
\end{align}
\end{proof}

\begin{prop}\label{prop:algebraicpart}
We suppose that $S_{p}\in \bbZ$. Under the same assumptions as in Lemma \ref{lem:CRT}, the constant $S_{p}$ is equal to $0$ if and only if $p$ divides $L_{E, k}$.
\end{prop}
\begin{proof}
The estimate (\cite[Proposition 2]{VillegasZagier}) of the real number $\absolute{L(E_{p}/\bbQ, 1)}$ yields $|S_{p}|<p$. Thus we only need to show that  $S_{p}$ is congruent to $0$ modulo $p$ if and only if $L_{E, k}$ is  congruent to $0$ modulo $p$.

It is straight forward to check
\begin{align}
&L(\psi\chi, 1)=\left.\sum_{(\fraka, 4p)=1}\chi(\fraka)\dfrac{1}{\overline{\psi}(\fraka)N\fraka^s}\right|_{s=0},\label{eq:Specialvalue1}\\
&L(\psi^{2k-1}, k)=\left.\sum_{(\fraka, 4)=1}\skakko{\dfrac{\alpha}{\overline{\alpha}}}^{k-1}\dfrac{1}{\overline{\psi}(\fraka)N\fraka^s}\right|_{s=0}.\label{eq:Specialvalue2}
\end{align}
We set $\varepsilon_1(\fraka)=\chi(\fraka)\psi(\fraka), \varepsilon_2(\fraka)=(\psi(\fraka)/\bar{\psi}(\fraka))^{k-1}\psi(\fraka)$. Since the elliptic curve $E_{1}$ is ordinary at $p$, there exists a $p$-adic $L$-function interpolating special values \eqref{eq:Specialvalue1} and \eqref{eq:Specialvalue2}. We denote $L_\frakf(\varepsilon, s)$ by the Hecke $L$-function associated to a Hecke character $\varepsilon$ without the Euler factor at the primes dividing $\frakf$. Since the elliptic curve $E_p$ is defined over $\bbQ$, we have $L_{4p}(\varepsilon_1^{-1}, 0)=L(\psi\chi, 1)$ and $L_{4}(\varepsilon_2^{-1}, 0)=L(\psi^{2k-1}, k)$.

Let $(\Omega, \Omega_p)$ be the pair of complex period and $p$-adic period as in \cite[p. 68, DEFINITION]{deShalit} and let $\mu$ be the $p$-adic measure on $\calG={\rm Gal}(K(4p^{\infty})/K)$ related to the $p$-adic $L$-function of $E_1$. Then the following identities, both sides of which lie in $\bar{\bbQ}$, holds:
\begin{align}
&\dfrac{1}{\Omega_p}\int_{\calG}\varepsilon_1(\sigma)d\mu(\sigma)=\dfrac{1}{\Omega}G(\varepsilon_1)L_{4p}(\varepsilon_1^{-1}, 0),\\
&\dfrac{1}{\Omega_p^{2k-1}}\int_{\calG}\varepsilon_2(\sigma)d\mu(\sigma)=\dfrac{(k-1)!}{\Omega^{2k-1}}\varpi^{k-1}G(\varepsilon_2)\skakko{1-\dfrac{\varepsilon_2(\frakp)}{p}}^2L_{4}(\varepsilon_2^{-1}, 0),
\end{align}
where $G(\varepsilon)$ is a "Gauss sum" (see definition \cite[p. 80]{deShalit}). Lemma \ref{lem:CRT} shows
\begin{align}
\absolute{\int_{\calG}\varepsilon_1(\sigma)d\mu(\sigma)-\int_{\calG}\varepsilon_2(\sigma)d\mu(\sigma)}_{\pi}\leq \max_{(\fraka, 4p)=1}\absolute{\varepsilon_1(\fraka)-\varepsilon_2(\fraka)}_{\pi}\leq \dfrac{1}{p}.
\end{align}
Therefore we obtain the congruence relation
\begin{align}
\dfrac{\Omega_p}{\Omega}G(\varepsilon_1)L_{4p}(\varepsilon_1^{-1}, 0)\equiv \dfrac{\Omega_p^{2k-1}(k-1)!}{\Omega^{2k-1}}\varpi^{k-1}L_{4}(\varepsilon_2^{-1}, 0) \mod p.
\end{align}
By \cite[p. 91, Lemma]{deShalit} and \cite[p. 8, (14)]{Loxton}, $G(\varepsilon_1)^2$ is equal to $\sqrt{p}\bar{\pi}$ up to units in $\calO_K^\times$ and $G(\varepsilon_2)$ is equal to $1$. Moreover, (\cite[p, 9-10]{deShalit}) shows $\Omega_p^{p-1}\equiv \bar{\pi}^{-1}\bmod p$. Hence it follows that
\begin{align}
\bar{\pi}S_{p}\equiv u2^{4k-5}3^{3k-3}L_{E, k}\mod p\label{eq:equivalencerelation}
\end{align}
for some $u\in \calO_K^\times$. The assertion follows from this.
\end{proof}

\begin{rem}
It is known that $(\frac{p-1}{2})!^2\equiv -1\bmod p$ and \cite[Corollary 6.6]{Lemmermeyer} shows
\begin{align}
\binom{\frac{p-1}{2}}{\frac{p-1}{4}}\equiv \pi+\bar{\pi} \mod p.
\end{align}
Thus \eqref{eq:equivalencerelation} can be rewritten as
\begin{align}
S_p\equiv \pm \skakko{\dfrac{p-1}{4}}!^22^{4k-5}3^{3k-3}L_{E, k} \mod p.
\end{align}
The proof of Proposition \ref{prop:algebraicpart} essentially shows Villegas' and Zagier's congruence relation \cite[p. 7]{VillegasZagier}
\begin{align}
S_{A, p}\equiv (-3)^{(p-10)/3}\skakko{\dfrac{p-1}{3}}!^2L_{A, k} \mod p,
\end{align}
where $S_{A, p}$ is the algebraic part of the special value $L(A_p/\bbQ, 1)$. The algebraic number $L_{A, k}$ is explained in detail below.
\end{rem}

By Proposition \ref{prop:algebraicpart}, we only need to calculate the algebraic part $L_{-4, k}$. (Actually, $L_{-4, k}$ is a square of a rational integer. We calculate the square root of it. ) 

Let $\psi'$ be the Hecke character of $\bbQ(\sqrt{-3})$ associated to $A_{1}: x^3+y^3=1$. We define the algebraic part of $L(\psi'^{2k-1}, k)$ to be
\begin{align}
L_{A, k}=3\nu \left(\dfrac{2\varpi}{2\sqrt{3}\Omega_{A}^2} \right)^{k-1}\dfrac{(k-1)!}{\Omega_{A}}L(\psi'^{2k-1}, k),
\end{align}
where $\Omega_{A}=\Gamma(1/3)^3/(2\varpi\sqrt{3})$ is the real period of $E_{1}$ and $\nu=2$ if $k\equiv 2\bmod 6$, $\nu=1$ otherwise. For the case where $p$ is congruent to $1$ modulo $9$, we see that the rank of $A_{p}$ is equal to $0$ if and only if $p$ divides $L_{A, k}$ in the same way for $E_{p}$.

\subsection{Maass-Shimura operator}

Unless otherwise stated, we denote by $\Gamma\subset SL_2(\bbR)$ a congruence subgroup. Let $M_k(\Gamma)$ be the space of holomorphic modular forms of weight $k$ for $\Gamma$. In general, $M_k^\ast(\Gamma)$ denotes the space of differentiable modular form, possibly with some character or multiplier system. Let $D$ be the differential operator
\begin{align}
D=\dfrac{1}{2\varpi i}\dfrac{d}{dz}=q\dfrac{d}{dq} \quad(q=e^{2\varpi iz}).
\end{align}
The Maass-Shimura operator
\begin{align}
\partial_k=D-\dfrac{k}{4\varpi y} \quad(z=x+iy),
\end{align}
preserves a modular relation. This is because the Maass-Shimura operator is compatible with the slash operator, that is, the following holds:
\begin{align}
{}^\forall \gamma\in SL_2(\bbR),\, \partial_k(f|[\gamma]_k)=(\partial_kf)|[\gamma]_{k+2}.\label{eq:compatiblity}
\end{align}
Moreover if $f\in M_k^\ast(\Gamma)$, then $\partial_k^{(h)}f\in M_{k+2h}^\ast(\Gamma)$, where
\begin{align}
\partial_k^{(h)}=\partial_{k+2h-2}\circ \partial_{k+2h-4}\circ \dots \circ \partial_{k+2}\circ \partial_{k}.
\end{align}

\begin{prop}[{\cite[p.4, (16)]{Villegas}}]\label{prop:MSEisenstein}
\begin{align}
\partial_k^{(h)}\left(\dfrac{1}{(mz+n)^k} \right)=\dfrac{(h+k-1)!}{(k-1)!}\left(\dfrac{-1}{4\varpi y}\dfrac{m\bar{z}+n}{mz+n} \right)^h\dfrac{1}{(mz+n)^k}.
\end{align}
\end{prop}

We define the $h$-th generalized Laguerre polynomial to be
\begin{align}
L_h^\alpha(z)=\sum_{j=0}^\infty \binom{h+\alpha}{h-j}\dfrac{(-z)^j}{j!} \quad (h\in \bbZ_{\geq 0},\,\alpha\in \bbC).
\end{align}
In the special case $\alpha=1/2,\,-1/2$, we see that
\begin{align}
H_{2n}(z)=(-4)^nn!L_n^{-1/2}(z^2),\quad H_{2n+1}(z)=2(-4)^nn!zL_n^{1/2}(z^2),\label{eq:LaguerreHermite}
\end{align}
where
\begin{align}
H_n(z)=\sum_{0\leq j \leq n/2}\dfrac{n!}{j!(n-2j)!}(-1)^j(2z)^{n-2j}
\end{align}
is the $n$-th Hermite polynomial. 

\begin{prop}[{\cite[p.3, (9)]{Villegas}}]\label{prop:MSq}
The following holds.
\begin{align}
\partial_k^{(h)}\left(\sum_{n=0}^\infty a(n)e^{2\varpi i nz} \right)=\dfrac{(-1)^hh!}{(4\varpi y)^h}\sum_{n=0}^\infty a(n)L_h^{k-1}(4\varpi ny)e^{2\pi inz}.
\end{align}
In particular for $k=1/2, 3/2$, we have
\begin{align}
&\partial_{1/2}^{(h)}\left(\sum_{n=0}^\infty a(n)e^{\varpi in^2z} \right)=\dfrac{(-1)^hh!}{(4\varpi y)^h}\sum_{n=0}^\infty a(n)L_h^{-1/2}(2n^2\varpi y)e^{\varpi in^2z},\\
&\partial_{3/2}^{(h)}\left(\sum_{n=0}^\infty a(n)e^{\varpi in^2z} \right)=\dfrac{(-1)^hh!}{(4\varpi y)^h}\sum_{n=0}^\infty a(n)L_h^{1/2}(2n^2\varpi y)e^{\varpi in^2z}.
\end{align}
\end{prop}

We introduce the following theta series, whose notation is based on \cite{FarkasKra}.

\begin{align}
&\theta\character{\epsilon}{\epsilon'}(z,\tau):=\sum_{n\in \bbZ}\exp 2\varpi i\mkakko{\dfrac{1}{2}\left(n+\dfrac{\epsilon}{2} \right)^2\tau+\left(n+\dfrac{\epsilon}{2} \right)\left(z+\dfrac{\epsilon'}{2} \right)}\quad (\epsilon, \epsilon'\in \bbQ),\label{eq:FKtheta1}\\
&\theta'\character{\epsilon}{\epsilon'}(0, \tau):=\dfrac{\partial}{\partial z}\left.\theta\character{\epsilon}{\epsilon'}(z, \tau)\right|_{z=0}=2\varpi i\sum_{n\in \bbZ}\skakko{n+\dfrac{\epsilon}{2}} \exp 2\varpi i\mkakko{\dfrac{1}{2}\skakko{n+\dfrac{\epsilon}{2}}^2\tau+\dfrac{\epsilon'}{2}\skakko{n+\dfrac{\epsilon}{2}}}\label{eq:FKtheta2}.
\end{align}
The action of the Maass-Shimura operator on \eqref{eq:FKtheta1} and \eqref{eq:FKtheta2} is described by
\begin{align}
&\theta_{(p)}\character{\mu}{\nu}(z):=i^{-p}(2\varpi y)^{-p/2}\sum_{n\in \bbZ+\mu}H_p(n\sqrt{2\varpi y})\exp(\varpi in^2z+2\varpi i\nu n)\quad (\mu, \nu\in \bbQ, p\in \bbZ_{\geq 0}).
\end{align}

\begin{prop}\label{prop:partialtheta}
For $h\in \bbZ_{\geq 0}$, it holds that
\begin{align}
&\theta_{(2h)}\character{\mu}{\nu}(z)=(-1)^h2^{3h}\partial_{1/2}^{(h)}\skakko{\theta\character{2\mu}{2\nu}(0, z)},\\
&\theta_{(2h+1)}\character{\mu}{\nu}(z)=-i(-1)^{h}2^{3h+1}\partial_{3/2}^{(h)}\skakko{\dfrac{1}{2\varpi i}\theta'\character{2\mu}{2\nu}(0, z)}.
\end{align}
\end{prop}
\begin{proof}
It follows by Proposition \ref{prop:MSq} and the identities \eqref{eq:LaguerreHermite}.
\end{proof}

\subsection{The special value of $L$-function with the Maass-Shimura operator}

Let $\psi$ be the Hecke character of $K=\bbQ(i)$ associated to $E_{1}: y^2=x^3+x$, where $i=\sqrt{-1}$. For an integral ideal $\fraka$ of $\calO_K$ which is prime to $4$, we have
\begin{align}
\psi(\fraka)=(-1)^{(a-1)/2}(a+bi),
\end{align}
where $a+bi$ is the primary generator, that is, $a+bi$ satisfies $(a, b)\equiv (1, 0), (3, 2)\bmod 4$. We set $\varepsilon(a+bi)=(-1)^{(a-1)/2}$.

\begin{lem}\label{lem:idealform}
An integral ideal $\fraka$ of $\calO_K$ which is prime to $4$ is written in the form
\begin{align}
\fraka=(r+4N-2mi) \quad (r\in \{1, 3\}, N, m\in \bbZ).
\end{align}
\end{lem}
\begin{proof}
An ideal $(a+bi)$ is prime to $4$ if and only if the norm $a^2+b^2$ is prime to $4$. Therefore such an ideal $(a+bi)$ must satisfy $(a, b)\equiv (1, 0), (0, 1)\bmod 2$. There is nothing to prove the former case. For the latter case, it follows from $(a+bi)=(b-ai)$.
\end{proof}

Let $\Theta(z)$ be the theta series
\begin{align}
\Theta(z)=\sum_{\lambda\in \calO_K}q^{N_{K/\bbQ}\lambda}=\sum_{n, m\in \bbZ}q^{n^2+m^2}\in M_1(\Gamma_1(4)).
\end{align}

\begin{prop}
We have
\begin{align}
L(\psi^{2k-1}, k)=\dfrac{(-1)^{k-1}2^{-3}\varpi^k}{(k-1)!}\skakko{\partial_1^{(k-1)}\Theta(z)|_{z=i/4}+\partial_1^{(k-1)}\Theta(z)|_{z=i/4+1/2}}.
\end{align}
\end{prop}
\begin{proof}
We consider the Eisenstein series of weight $1$ for $\Gamma_1(4)$
\begin{align}
G_{1, \varepsilon}(z)=\lim_{s\to 0}\dfrac{1}{2}\sum_{n, m}\dfrac{\varepsilon(n)}{(4mz+n)\absolute{4mz+n}^{2s}} \quad(z\in \bbH),
\end{align}
where $\sum'$ implies that $(n, m)=(0, 0)$ is excluded. By using Proposition \ref{prop:MSEisenstein}, we have
\begin{align}
\partial_1^{(k-1)}G_{1, \varepsilon}(z)=(k-1)!\skakko{\dfrac{-1}{4\varpi y}}^{k-1}\dfrac{1}{2}\sum_{n, m}\dfrac{\varepsilon(n)(n+4m\bar{z})^{2k-1}}{|n+4mz|^{2k}}.
\end{align}
Since $G_{1, \varepsilon}(z)=\varpi/4 \cdot \Theta(z)$ (Note that $\dim M_1(\Gamma_1(4))=1$.), it holds that
\begin{align}
L(\psi^{2k-1}, k)&=\sum_{r, N, m}\dfrac{\psi((r+4N-2mi))^{2k-1}}{|r+4N-2mi|^{2k}}\nonumber \\
&=\dfrac{1}{2}\sum_{r, N, m}\dfrac{\varepsilon(r+4N)(r+4N-2mi)^{2k-1}}{|r+4N+2mi|^{2k}}\nonumber \\
&=\dfrac{1}{2}\sideset{}{^{'}}\sum_{n, m}\dfrac{\varepsilon(n)(n-2mi)^{2k-1}}{|n+2mi|^{2k}}\nonumber \\
&=\dfrac{(-1)^{k-1}2^{k-3}\varpi^k}{(k-1)!}\partial_1^{(k-1)}\Theta(z)|_{z=i/2},\label{eq:forevenkcorollary}
\end{align}
Finally the identity (\cite[p.192]{Kohler})
\begin{align}
2\Theta(z)=\Theta\skakko{\dfrac{z}{2}}+\Theta\skakko{\dfrac{z+1}{2}}
\end{align}
yields the claim.
\end{proof}

\begin{cor}
If $k$ is an even integer, then $L(\psi^{2k-1}, k)=0$.
\end{cor}
\begin{proof}
For $\gamma=\skakko{\begin{array}{cc} 0 & -1\\ 4 & 0\end{array}}\in GL_2^+(\bbQ)$, we have $\Theta(z)|[\gamma]_1=-i\Theta(z)$ (\cite[p. 124]{Koblitz}). By \eqref{eq:compatiblity}, we have
\begin{align}
\partial_1^{(k-1)}\Theta(z)=i(2z)^{-2k+1}\partial_1^{(k-1)}\Theta(z)|_{z=-1/4z}.
\end{align}
Thus we obtain $\partial_1^{(k-1)}\Theta(z)|_{z=i/2}=0$ and the colollary follows by \eqref{eq:forevenkcorollary}.
\end{proof}

Next, we write the special value $L(\psi^{2k-1}, k)$ as a square of the $\partial_k$-derivative of some modular form. The key is Theorem \ref{thm:FF} below. Note that by Proposition \ref{prop:MSq}, it holds that
\begin{align}
\partial_{1}^{(k-1)}\Theta(z)&|_{z=i/4}+\partial_1^{(k-1)}\Theta(z)|_{z=i/4+1/2}\nonumber \\
=&2\dfrac{(-1)^{k-1}(k-1)!}{\varpi^{k-1}}\sum_{(0, 0), (1, 1)}L_{k-1}^0(2\varpi Q_i(n, m))e^{-\varpi(n^2+m^2)/2},
\end{align}
where $\sum_{(a, b)}$ implies that $(n, m)$ runs over all pairs of integer which satisfy $(n, m)\equiv (a, b)\bmod 2$. We set 
\begin{align}
a_{n, m}:=L_{k-1}^0(2\pi Q_i(n, m))e^{-\varpi(n^2+m^2)/2}.
\end{align}

\begin{thm}[{\cite[p.7]{Villegas}}]\label{thm:FF}
For $a\in \bbZ_{>0},\, z\in \bbH,\, \mu, \nu\in \bbQ,\, p, \alpha\in \bbZ_{\geq 0}$, the following identity holds. 
\begin{multline*}
\dfrac{(-1)^pp!}{(\varpi y)^p}\sum_{n, m\in \bbZ}e^{2\varpi i(n\mu+m\nu)}\left(\dfrac{mz-n}{ay} \right)^\alpha L_p^\alpha\left(\dfrac{2\varpi}{a}Q_z(n, m) \right)e^{\varpi (inm-Q_z(n, m))/a}\\
=\sqrt{2ay}(ay)^\alpha \theta_{(p)}\character{a\mu}{\nu}(a^{-1}z)\theta_{(p+\alpha)}\character{\mu}{-a\nu}(-a\bar{z}).
\end{multline*}
In particular for the case $a=1,\,\alpha=0$, the right hand side is
\begin{align}
(-1)^p\sqrt{2y}\left|\theta_{(p)}\character{\mu}{\nu}(z) \right|^2.
\end{align}
\end{thm}

We define $\theta_2, \theta_4$ to be
\begin{align*}
\theta_2(z):=\theta\character{1}{0}(0, z)=\sum_{n\in \bbZ+1/2}e^{\varpi in^2z},\quad \theta_4(z):=\theta\character{0}{1}(0, z)=\sum_{n\in \bbZ}(-1)^ne^{\varpi in^2z}.
\end{align*}

\begin{thm}\label{thm:mainthm3}
Let $\psi$ be the Hecke character of $K=\bbQ(i)$ associated to $E_{1}: y^2=x^3+x$. Then for $L(\psi^{2k-1}, s)$, we have 
\begin{align}
L(\psi^{2k-1}, k)=\begin{cases}
\dfrac{2^{3k-9/2}\varpi^k}{(k-1)!}\absolute{\partial_{1/2}^{(N)}\theta_2(z)|_{z=i}}^2 & (k=2N+1),\\
0 & (k=2N).
\end{cases}
\end{align}
\end{thm}
\begin{proof}
We apply for $p=k-1, a=1,\,\alpha=0, z=i$ in Theorem \ref{thm:FF}. By substituting $(\mu,\,\nu)=(1/2,\,0),\,(0,\,1/2)$, we see that
\begin{align}
&\dfrac{(k-1)!}{\varpi^{k-1}}\skakko{\sum_{(0, 0), (0, 1), (1, 1)}a_{n, m}-\sum_{(1, 0)}a_{n, m}}=\sqrt{2}\absolute{\theta_{(k-1)}\character{1/2}{0}(i)}^2,\label{eq:FF1}\\
&\dfrac{(k-1)!}{\varpi^{k-1}}\skakko{\sum_{(0, 0), (1, 0), (1, 1)}a_{n, m}-\sum_{(0, 1)}a_{n, m}}=\sqrt{2}\absolute{\theta_{(k-1)}\character{0}{1/2}(i)}^2.\label{eq:FF2}
\end{align}
Note that
\begin{align}
\absolute{\theta_{(k-1)}\character{1/2}{0}(z)}^2= \absolute{\theta_{(k-1)}\character{0}{1/2}(z)}^2.
\end{align}
By adding \eqref{eq:FF1} and \eqref{eq:FF2}, we obtain
\begin{align}
\partial_1^{(k-1)}\Theta(z)|_{z=i/4}+\partial_1^{(k-1)}\Theta(z)|_{i/4+1/2}=(-1)^{k-1}2^{3/2}\absolute{\theta_{(k-1)}\character{1/2}{0}(i)}^2.
\end{align}
Therefore the theorem follows by Proposition \ref{prop:partialtheta}.
\end{proof}

\begin{cor}
Under the same condition as Theorem \ref{thm:mainthm3}, we have
\begin{align}
L(\psi^{2k-1}, k)\geq 0.
\end{align}
\end{cor}

\subsubsection{The case for $A_p$}

Let $\psi'$ be the Hecke character of $K=\bbQ(\omega)$ associated to $A_{1}: x^3+y^3=1$, where $\omega=(-1+\sqrt{-3})/2$. For an integral ideal $\fraka$ of $\calO_K$ which is prime to $3$, we have
\begin{align}
\psi'(\fraka)=\psi'((a+bi))=\varepsilon'(a+bi)(a+bi),
\end{align}
where $\varepsilon': \skakko{\calO_K/3\calO_K}^\times \to \bbC^\times$ is some sextic character.

\begin{lem}
An integral ideal $\fraka$ of $\calO_K$ which is prime to $3$ is written in the form
\begin{align}
\fraka=(r+3(N+m\omega^2)) \quad(r\in \{1, 2 \}, N, m\in \bbZ),
\end{align}
\end{lem}
\begin{proof}
A proof is the same as Lemma \ref{lem:idealform}.
\end{proof}

Let $\Theta'(z)$ be the theta series
\begin{align*}
\Theta'(z)=\sum_{\lambda\in \calO_K}q^{N\lambda}=\sum_{n, m}q^{n^2+nm+m^2}\in M_1(\Gamma_1(3)).
\end{align*}

\begin{prop}
We have
\begin{align}
L(\psi'^{2k-1}, k)=\dfrac{(-1)^{k-1}2^{k-1}3^{-k/2-2}\varpi^k}{(k-1)!}\omega^{k-1}(1-\omega)\partial_1^{(k-1)}\Theta'(z)|_{z=(\omega-2)/3}.
\end{align}
\end{prop}
\begin{proof}
Similarly for the case $E_{p}$, we obtain
\begin{align}
L(\psi'^{2k-1}, k)&=\dfrac{1}{2}\sideset{}{^{'}}\sum_{n, m}\dfrac{\varepsilon'(n)(n+3m\omega^2)^{2k-1}}{|n+3m\omega|^{2k}}\\
&=\dfrac{(-1)^{k-1}2^{k-1}3^{k/2-2}\varpi^k}{(k-1)!}\partial_1^{(k-1)}\Theta'(z)|_{z=\omega}.
\end{align}
For the Atkin-Lehner involution $W_3=\begin{pmatrix} 0 & -1/\sqrt{3} \\ \sqrt{3} & 0 \end{pmatrix}$, we have $\Theta'(z)|[W_3]_1=-i\Theta'(z)$ (\cite[p.155]{Kohler}). By \eqref{eq:compatiblity}, we have
\begin{align}
\partial_1^{(k-1)}\Theta'(z)=i(\sqrt{3}z)^{-2k+1}\partial_1^{(k-1)}\Theta'(z)|_{z=-1/3z}.
\end{align}
The proposition follows by substituting $z=\omega$.
\end{proof}

By Proposition \ref{prop:MSq}, it holds that
\begin{align}
\partial_1^{(k-1)}&\Theta(z)|_{z=(\omega-2)/3}\nonumber\\
=&\dfrac{(-1)^{k-1}\sqrt{3}^{k-1}(k-1)!}{2^{k-1}\varpi^{k-1}}\sum_{n, m\in \bbZ}L_{k-1}^0(2\varpi Q_\omega(n, m))e^{2\varpi i(n^2+nm+m^2)(\omega-2)/3}.
\end{align}
We set
\begin{align}
a_{n, m}:=L_{k-1}^0(2\varpi Q_\omega(n, m))e^{2\varpi i(n^2+nm+m^2)(\omega-2)/3}.
\end{align}

\begin{lem}\label{lem:mainlem}
For $h, N\in \bbZ_{\geq 0}$, the following holds.
\begin{itemize}
\item[$(1)$] $\partial_{1/2}^{(h)}\left.\theta \character{1/3}{-1/3}(0, z)\right|_{z=\omega}~=e^{h\varpi i/3-\varpi i/4}3^{1/4}\left.\partial_{1/2}^{(h)}\eta(z)\right|_{z=\omega}$,\\
\item[$(2)$] $\partial_{3/2}^{(h)}\left.\dfrac{1}{2\varpi i}\theta'\character{1}{1}(z)\right|_{z=\omega}=e^{\varpi i/2}\partial_{3/2}^{(h)}\eta(z)^3|_{z=\omega}$,\\
\item[$(3)$] $\partial_{3/2}^{(3N+1)}\left.\dfrac{1}{2\varpi i}\theta'\character{1/3}{-1/3}(z)\right|_{z=\omega}=e^{N\varpi i-13\varpi i/36}2^{-1}3^{5/4}\partial_{3/2}^{(3N+1)}\eta(3z)^3|_{z=\omega}$.
\end{itemize}
\end{lem}
\begin{proof}
(1)~By using identity (\cite[p. 241]{FarkasKra})
\begin{align}
\theta\character{1/3}{1}(0, z)=e^{\varpi i/6}\eta(z),
\end{align}
we have
\begin{align}
\theta\character{1/3}{-1/3}(0, z)=e^{-7\varpi i/36}\theta\character{1/3}{1}(0, z)=e^{-\varpi i/36}\eta\skakko{\dfrac{z-1}{3}}
\end{align}
It follows from this and \eqref{eq:compatiblity}.

(2) It follows from the identity \cite[p.289, (4.14)]{FarkasKra}
\begin{align}
\theta'\character{1}{1}(0, \tau)=-2\varpi \eta(\tau)^3.\label{eq:etaidentity}
\end{align}

(3)~By \eqref{eq:compatiblity}, we have
\begin{align}
\partial_{3/2}^{(3N+1)}\eta(z)^3|_{z=\omega}=0.\label{eq:partialetazero}
\end{align}
It follows from \eqref{eq:etaidentity}, \eqref{eq:partialetazero} and the identity \cite[p.240, (3.40)]{FarkasKra}
\begin{align}
6e^{\varpi i/3}\theta'\character{1/3}{1}(0, 3z)&=\theta'\character{1}{1}(0, z/3)+3\theta'\character{1}{1}(0, 3z).
\end{align}
\end{proof}

\begin{thm}\label{thm:mainthm4}
Let $\psi'$ be the Hecke character of $K=\bbQ(\omega)$ associated to $A_1: x^3+y^3=1$. Then for $L(\psi'^{2k-1}, s)$, we have
\begin{align}
L(\psi'^{2k-1}, k)=\begin{cases}
\dfrac{\varpi^k}{(k-1)!}2^{2k-1}3^{k/2-9/4}\absolute{\partial_{1/2}^{(3N)}\eta(z)|_{z=\omega}}^2 & (k=6N+1),\\
\dfrac{\varpi^k}{(k-1)!}2^{2k-3}3^{k/2-11/4}\absolute{\partial_{3/2}^{(3N+1)}\eta(z)^3|_{z=\omega}}^2  & (k=6N+2),\\
\dfrac{\varpi^k}{(k-1)!}2^{2k-4}3^{k/2-1/4}\absolute{\partial_{3/2}^{(3N+1)}\eta(3z)^3|_{z=\omega}}^2 & (k=6N+4),\\
0 & ({\rm otherwise}).
\end{cases}
\end{align}
\end{thm}
\begin{proof}
We apply for $p=k-1,\, a=1,\,\alpha=0,\, z=\omega$ in Theorem \ref{thm:FF}. By substituting $(\mu, \nu)=(1/2, 1/2)$ with multiplication $\omega^2$, $(\mu, \nu)=(1/6, -1/6)$ and $(-1/6, 1/6)$, we see that
\begin{align}
&\dfrac{2^{k-1}(k-1)!}{\sqrt{3}^{k-1}\varpi^{k-1}}\left(\sum_{n-m\equiv 1, 2, 4, 5}a_{n, m}+\sum_{n-m\equiv 0, 3}a_{n, m} \right)=\omega^2\sqrt[4]{3}\left|\theta_{(k-1)}\left[\begin{array}{c}1/2 \\ 1/2 \end{array}\right](\omega) \right|^2,\label{eq:FF3}\\
&\dfrac{2^{k-1}(k-1)!}{\sqrt{3}^{k-1}\varpi^{k-1}}\left(\sum_{n-m\equiv 0, 1, 3, 4}a_{n, m}+\sum_{n-m\equiv 2, 5}a_{n, m} \right)=\sqrt[4]{3}\left|\theta_{(k-1)}\left[\begin{array}{c}1/6 \\ -1/6 \end{array}\right](\omega) \right|^2,\label{eq:FF4}\\
&\dfrac{2^{k-1}(k-1)!}{\sqrt{3}^{k-1}\varpi^{k-1}}\left(\sum_{n-m\equiv 0, 2, 3, 5}a_{n, m}+\sum_{n-m\equiv 1, 4}a_{n, m} \right)=\sqrt[4]{3}\left|\theta_{(k-1)}\left[\begin{array}{c}-1/6 \\ 1/6 \end{array}\right](\omega) \right|^2,\label{eq:FF5}
\end{align}
where $\sum_{n-m\equiv a}$ implies that $(n, m)$ runs over all pairs of integer which satisfy $n-m\equiv a\bmod 6$. Note that
\begin{align}
\left|\theta_{(p)}\left[\begin{array}{c}\mu \\ -\nu \end{array}\right](z) \right|^2=\left|\theta_{(p)}\left[\begin{array}{c}-\nu \\ \mu \end{array}\right](z) \right|^2.
\end{align}
By adding \eqref{eq:FF3}, \eqref{eq:FF4} and \eqref{eq:FF5}, we obtain
\begin{align}
L(\psi'^{2k-1}, k)=\dfrac{2^{-k+1}3^{k/2-11/4}\varpi^k}{(k-1)!}\mkakko{\omega^{k+1}\absolute{\theta_{(k-1)}\character{1/2}{1/2}(\omega)}^2+2\omega^{k-1}\absolute{\theta_{(k-1)}\character{1/6}{-1/6}(\omega)}^2}.
\end{align}
Since $L(\psi'^{2k-1}, k )$ takes a real number, it holds that
\begin{align}
L(\psi'^{2k-1}, k)=\begin{cases}
0 & (k\equiv 0, 3\mod 6),\\
\dfrac{2^{-k+2}3^{k/2-11/4}\varpi^k}{(k-1)!}\absolute{\theta_{(k-1)}\character{1/6}{1/6}(\omega)}^2 & (k\equiv 1, 4\mod 6),\\
\dfrac{2^{-k+1}3^{k/2-11/4}\varpi^k}{(k-1)!}\absolute{\theta_{(k-1)}\character{1/2}{1/2}(\omega)}^2 & (k\equiv 2, 5\mod 6).
\end{cases}
\end{align}
Since 
\begin{align}
\theta\character{1}{1}(0,z)=0,
\end{align}
the theorem follows by Proposition \ref{prop:partialtheta} and Lemma \ref{lem:mainlem}.
\end{proof}

\begin{cor}
Under the same condition as Theorem \ref{thm:mainthm4}, we have 
\begin{align}
L(\psi'^{2k-1}, k)\geq 0.
\end{align}
\end{cor}

\section{Recurrence formula}

In this section, we also denote $\varpi=3.1415\cdots$ by the real number. The Maass-Shimura operator does not map a modular form to a modular form in general, but the Ramanujan-Serre operator
\begin{align}
\vartheta_k=D-\dfrac{k}{12}E_2
\end{align}
does, where $E_2(z)=1-24\sum_{n=1}^\infty \sigma_1(n)q^n$ is the Eisenstein series of weight $2$. This Eisenstein series is not a modular form, but the function $E_2^\ast(z)=E_2(z)-3/\varpi y$ is non-holomorphic modular form. Since the Ramanujan-Serre operator is also expressed as $\vartheta_k=\partial_k-kE_2^\ast/12$, we see that $\vartheta_k$ maps a modular form of weight $k$ to a modular form of weight $k+2$. To express the difference between $\partial_k$ and $\vartheta_k$, Villegas and Zagier have introduced the following series.
\begin{align}
&f_\partial(z, X):=\sum_{n=0}^\infty \dfrac{\partial_k^{(n)}f(z)}{k(k+1)\dots (k+n-1)}\dfrac{X^n}{n!} \quad(z\in \bbH,\,X\in\bbC,\,f\in M_k(\Gamma))\nonumber\\
&f_\vartheta(z, X):=e^{-E_2^\ast(z)X/12}f_\partial(z, X)\label{eq:recurrenceseries}
\end{align}

\begin{prop}[{\cite[p. 12]{VillegasZagier}}]\label{prop:recurrence}
Let $f\in M_k(\Gamma)$. Then the series $f_\vartheta(z, X)$ has the expansion
\begin{align}
f_\vartheta(z, X)=\sum_{n=0}^\infty \dfrac{F_n(z)}{k(k+1)\dots (k+n-1)}\dfrac{X^n}{n!},
\end{align}
where $F_n\in M_{k+2n}(\Gamma)$ is the modular form that is defined by the following recurrence formula.
\begin{align}
F_{n+1}=\vartheta_{k+2n} F_n-\dfrac{n(n+k-1)}{144}E_4F_{n-1}\label{eq:recurrence}
\end{align}
The initial condition is $F_0=f,\,F_1=\vartheta_k f$.
\end{prop}

If a CM point $z_0$ satisfy $E_2^\ast(z_0)=0$, then $f_\partial(z_0, X)=f_\vartheta(z_0, X)$ by \eqref{eq:recurrenceseries}. Therefore by Proposition \ref{prop:recurrence}, we see that 
\begin{align}
\partial_k^{(n)}f(z)|_{z=z_0}=F_n(z_0),
\end{align}
where $F_n$ is the modular form that is defined by the recurrence formula \eqref{eq:recurrence}.

We apply Proposition \ref{prop:recurrence} for $f=\theta_2, \Gamma=\Gamma(2)$. The graded ring $\oplus_{k\in \frac{1}{2}\bbZ}M_k(\Gamma(2))$ is isomorphic to $\bbC[\theta_2, \theta_4]$ as $\bbC$-algebra ({\it cf}. \cite[p.28-29]{Zagier123}). Since $\theta_2$ and $\theta_4$ is algebraically independent over $\bbC$, we sometimes regard $\theta_2$ and $\theta_4$ as indeterminates and $\bbC[\theta_2, \theta_4]$ as the polynomial ring in two variables over $\bbC$.

\begin{lem}\label{lem:mainlem5}
We have
\begin{align}
\vartheta \theta_2=\dfrac{1}{12}\theta_2\theta_4^4+\dfrac{1}{24}\theta_2^5, ~\vartheta\theta_4=-\dfrac{1}{12}\theta_2^4\theta_4-\dfrac{1}{24}\theta_4^5
\end{align}
\end{lem}
\begin{proof}
It follows from the fact that $\vartheta \theta_2^4$ and $\vartheta \theta_4^4$ are of weight $4$ and the ring $M_4(\Gamma(2))$ is generated by $\theta_2^4, \theta_4^4$. 
\end{proof}

By the above lemma, the Ramanujan-Serre operator $\vartheta$ acts on $\bbC[\theta_2, \theta_4]$ as
\begin{align}
\vartheta=\skakko{\dfrac{1}{12}\theta_2\theta_4^4+\dfrac{1}{24}\theta_2^5}\dfrac{\partial}{\partial \theta_2}-\skakko{\dfrac{1}{12}\theta_2^4\theta_4+\dfrac{1}{24}\theta_4^5}\dfrac{\partial}{\partial \theta_4}.
\end{align}

\begin{thm}\label{thm:mainthm5}
We define the algebraic part of $L(\psi^{2k-1}, k)$ to be
\begin{align}
L_{E, k}=\dfrac{2^{k+1}3^{k-1}\varpi^{k-1}(k-1)!}{\Omega_{E}^{2k-1}}L(\psi^{2k-1}, k).
\end{align}
Then $L_{E, k}$ is a square of a rational integer and 
\begin{align}
\sqrt{L_{E, k}}=\begin{cases}
\absolute{f_N(0)} & (k=2N+1),\\
0 & (k=2N),
\end{cases}
\end{align}
where $f_n(t)\in\bbZ[t]$ is the polynomial that is defined by the recurrence formula
\begin{align}
f_{n+1}(t)=(4n+1)(2t+3)f_n(t)-12(t+1)(t+2)f_n'(t)-2n(2n-1)(t^2+3t+3)f_n(t)
\end{align}
The initial condition is $f_0(t)=1,\, f_1(t)=2t+3$.
\end{thm}
\begin{proof}
By Proposition \ref{prop:recurrence} and Lemma \ref{lem:mainlem5}, we have $\partial_{1/2}^{(n)}\theta_2(z)|_{z=i}=F_n(i)$, where $F_n$ is the modular form that is defined by the recurrence formula
\begin{align}
F_{n+1}=\skakko{\dfrac{1}{12}\theta_2\theta_4^2+\dfrac{1}{24}\theta_2^5}\dfrac{\partial F_n}{\partial \theta_2}-\skakko{\dfrac{1}{12}\theta_2^4\theta_4+\dfrac{1}{24}\theta_4^5}\dfrac{\partial F_n}{\partial \theta_4}-\dfrac{n(n-1/2)}{144}E_4F_{n-1}. \label{eq:step1ofmainthm5}
\end{align}
We set $f_n={24}^nF_n/\theta_2^{4n+1}$, which has degree $0$. Then we can rewrite the recurrence formula \eqref{eq:step1ofmainthm5} as follows:
\begin{align}
f_{n+1}=(4n+1)\dfrac{\theta_2^4+2\theta_4^4}{\theta_2^4}f_n+\dfrac{\theta_2^4+2\theta_4^4}{\theta_2^4}\dfrac{\partial f_n}{\partial \theta_2}-\dfrac{2\theta_2^4\theta_4+\theta_4^5}{\theta_2^4}\dfrac{\partial f_n}{\partial \theta_4}-2n(2n-1)\dfrac{E_4}{\theta_2^8}f_{n-1}.\label{eq:step2ofmainthm5}
\end{align}
Moreover we set $t=(\theta_4^4-\theta_2^4)/\theta_2^4$ which satisfies $t(i)=0$. Note that $E_4=\theta_2^8+\theta_2^4\theta_4^4+\theta_4^8$. Then the recurrence formula \eqref{eq:step2ofmainthm5} transforms
\begin{align}
f_{n+1}(t)=(4n+1)(2t+3)f_n(t)-12(t+1)(t+2)f_n'(t)-2n(2n-1)(t^2+3t+3)f_n(t).
\end{align}
The initial condition is $f_0(t)=1,\, f_1(t)=2t+3$. By the complex multiplication theory, we have
\begin{align}
\absolute{\theta_2(i)}=2^{-1/4}\varpi^{-1/2}\Omega_{E}^{1/2}.
\end{align}
Therefore we obtain
\begin{align}
\absolute{\partial_{1/2}^{(N)}\theta_2(z)|_{z=i}}^2=2^{-4k+7/2}3^{-k+1}\varpi^{-2k+1}\Omega_{E}^{2k-1}\absolute{f_N(0)}^2.
\end{align}
\end{proof}

\subsection{The case $A_p$}

First we consider the case for $k=6N+1$. (The case for $k=6N+2$ is almost the same. ) We apply Proposition \ref{prop:recurrence} for $f=\eta, \Gamma=\Gamma(1)$. The graded ring $\oplus_{k\in \bbZ}M_k(\Gamma(1))$ is isomorphic to $\bbC[E_4, E_6]$ as $\bbC$-algebra. Since $E_4$ and $E_6$ are algebraically independent over $\bbC$, we sometimes regard $E_4$ and $E_6$ as indeterminates and $\bbC[E_4, E_6]$ as the polynomial ring in two variables over $\bbC$. We denote by $\frac{\partial}{\partial E_4}$ and $\frac{\partial}{\partial E_6}$ the derivative with respect to formal variables $E_4$ and $E_6$. We take a sufficiently small neighborhood $D$ of $\omega$ so that $E_6^{1/3}$ can be defined. (Note that $E_6(\omega)\neq 0$.) In the following, we restrict the domain of functions in $\bbC[E_4, E_6, E_6^{1/3}, E_6^{-1}, \eta]$ to $D$.

\begin{lem}
We have
\begin{align}
\vartheta E_4=-\dfrac{1}{3}E_6,\quad \vartheta E_6=-\dfrac{1}{2}E_4^2,\quad \vartheta\eta=0.
\end{align}
\end{lem}
\begin{proof}
The proof is the same as Lemma \ref{lem:mainlem5}.
\end{proof}

By the above lemma, the Ramanujan-Serre operator $\vartheta$ acts on $\bbC[E_4, E_6]$ as
\begin{align}
\vartheta=-\dfrac{E_6}{3}\dfrac{\partial}{\partial E_4}-\dfrac{E_4^2}{2}\dfrac{\partial}{\partial E_6}.\label{eq:6N+1op}
\end{align}
The derivatives $\frac{\partial}{\partial E_4}$ and $\frac{\partial}{\partial E_6}$ on $\bbC[E_4, E_6]$ are uniquely extended on $\bbC[E_4, E_6, E_6^{-1}, E_6^{1/3}, \eta]$ satisfying the following:
\begin{align}
\dfrac{\partial}{\partial E_6} E_6^{-1}=-E_6^{-2}, \ \ \dfrac{\partial}{\partial E_6} E_6^{1/3}=\dfrac{1}{3}E_6^{-1}E_6^{1/3}.
\end{align}

Next we consider the case for $k=6N+4$. We apply Proposition \ref{prop:recurrence} for $f=\eta_3, \Gamma=\Gamma_0(3)$, where $\eta_3(z)=\eta(3z)^3$. It is known that the graded ring $\oplus_{k\in \bbZ}M_k(\Gamma_0(3))$ is isomorphic to $\bbC[C, \alpha, \beta]/(\alpha^2-C\beta)\cong \bbC[C, C^{-1}, \alpha]$ ({\it cf}. \cite{Suda}) as $\bbC$-algebra, where
\begin{align}
&C=\dfrac{1}{2}\skakko{3E_2(3z)-E_2(z)},\,\alpha=\dfrac{1}{240}\skakko{E_4(z)-E_4(3z)},\\
&\beta=\dfrac{1}{12}\mkakko{\dfrac{1}{504}\skakko{E_6(3z)-E_6(z)}-C\alpha}.
\end{align}
Since $C$ and $\alpha$ are algebraically independent over $\bbC$, we sometimes regard $C$ and $\alpha$ as indeterminates and $\bbC[C, \alpha]$ as the polynomial ring in two variables over $\bbC$. In the following, we consider the extension $\bbC[C, C^{-1}, \alpha, \eta_3]$ of $\bbC[C, \alpha]$.

\begin{lem}
We have
\begin{align}
\vartheta C=-\dfrac{1}{6}C^2+18\alpha, ~\vartheta \alpha=\dfrac{2}{3}C\alpha+9C^{-1}\alpha^2.
\end{align}
\end{lem}
\begin{proof}
The proof is the same as Lemma \ref{lem:mainlem5}.
\end{proof}
Similarly in the case for $k=6N+1$, the Ramanujan-Serre operator $\vartheta$ acts on $\bbC[C, C^{-1}, \alpha, \eta_3]$ as
\begin{align}
\vartheta=\skakko{-\dfrac{1}{6}C^2+18\alpha}\dfrac{\partial}{\partial C}+\skakko{\dfrac{2}{3}C\alpha+9C^{-1}\alpha^2}\dfrac{\partial }{\partial \alpha}.\label{eq:6N+3op}
\end{align}

\begin{thm}\label{thm:mainthm6}
We define the algebraic part of $L(\psi^{2k-1}, k)$ to be
\begin{align}
L_{A, k}=3\nu \skakko{\dfrac{2\varpi}{2\sqrt{3}\Omega_{A}^2}}^{k-1}\dfrac{(k-1)!}{\Omega_{A}}L(\psi'^{2k-1}, k)
\end{align}
where $\nu =2$ if $k\equiv 2\mod 6$, $\nu=1$ otherwise. Then $L_{A, k}$ is a square of a rational integer and
\begin{align}
\sqrt{L_{A, k}}=\begin{cases}
\absolute{x_{3N}(0)} & (k=6N+1),\\
\absolute{y_{3N}(0)} & (k=6N+2),\\
\absolute{z_{3N+1}(0)} & (k=6N+4),\\
0 & ({\rm otherwise}),
\end{cases}
\end{align}
where $x_{n}(t), y_{n}(t), z_{n}(t)\in \bbZ[t]$ are polynomials that is defined by the following recurrece formulas
\begin{align}
&x_{n+1}(t)=-2(1-8t^3)x_n'(t)-8nt^2x_n(t)-n(2n-1)tx_{n-1}(t)\\
&y_{n+1}(t)=-2(1-8t^3)y_n'(t)-8nt^2y_n(t)-n(2n+1)ty_{n-1}(t)\\
&z_{n+1}(t)=-(t-1)(9t-1)z_n'(t)+\mkakko{(6t-2)n+2}z_n(t)-2n(2n+1)tz_{n-1}(t).
\end{align}
The initial conditions are
\begin{align}
&x_0(t)=1,\, x_1(t)=0,\\
&y_0(t)=1,\, y_1(t)=0,\\
&z_0(t)=1/2,\, z_1(t)=1.
\end{align}
\end{thm}
\begin{proof}
Since the proof for the case $k=6N+2$ is the same for $k=6N+1$, we prove for the case $k=6N+1, 6N+4$. 

First we proof for $k=6N+1$. By Proposition \ref{prop:recurrence} and \eqref{eq:6N+1op}, we have $\partial_{1/2}^{(n)}\eta(z)|_{z=\omega}=X_n(\omega)$, where $X_n$ is the modular form that is defined by the recurrence formula
\begin{align}
X_{n+1}=-\dfrac{E_6}{3}\dfrac{\partial X_n}{\partial E_4}-\dfrac{E_4^2}{2}\dfrac{\partial X_n}{\partial E_6}-\dfrac{n(n-1/2)}{144}E_4X_{n-1}.\label{eq:step1ofmainthm6}
\end{align}
We set $x_n=12^nX_n/\eta E_6^{n/3}$ and $t=E_4E_6^{-2/3}/2$ which satisfies $t(\omega)=0$. Then we can rewrite the recurrence formula \eqref{eq:step1ofmainthm6} as follows:
\begin{align}
x_{n+1}(t)=-2(1-8t^3)x_n'(t)-8nt^2x_n(t)-n(2n+1)tx_{n-1}(t).
\end{align}
The initial condition is $x_0(t)=1,\,x_1(t)=0$. By the complex multiplication theory, we have
\begin{align}
\eta(\omega)=\dfrac{3^{3/8}\Omega_{A}^{1/2}}{2^{1/2}\varpi^{1/2}},\, E_6(\omega)=\dfrac{3^6\Omega_{A}^6}{2^3\varpi^6}
\end{align}
Therefore we have
\begin{align}
\absolute{\partial_{1/2}^{(3N)}\eta(z)|_{z=\omega}}^2=\dfrac{\Omega_{A}^{2k-1}}{\varpi^{2k-1}}2^{-3k+2}3^{k-1/4}\absolute{x_{3N}(0)}^2.
\end{align}

Next we proof for $k=6N+4$. We set $\eta_3(z)=\eta(3z)^3$. We have $\partial_{3/2}^{(n)}\eta_3(z)|_{z=\omega}=Z_n(\omega)$, where $Z_n$ is the modular form that is defined by the recurrence formula
\begin{align}
Z_{n+1}=\skakko{-\dfrac{1}{6}C^2+18\alpha}\dfrac{\partial Z_n}{\partial C}+\skakko{\dfrac{2}{3}C\alpha+9C^{-1}\alpha^2}\dfrac{\partial Z_n}{\partial \alpha}-\dfrac{n(n+1/2)}{144}E_4Z_{n-1}.\label{eq:step2ofmainthm6}
\end{align}
We set $z_n=2^{3n-1}Z_n/\eta_3C^n, t=(1+216C^{-2}\alpha)/9$, which satisfies $t(\omega)=0$. Then we can rewrite the recurrence formula \eqref{eq:step2ofmainthm6} as follows:
\begin{align}
z_{n+1}(t)=-(t-1)(9t-1)z_n'(t)+\mkakko{(6t-2)n+2}z_n(t)-2n(2n+1)tz_{n-1}(t).
\end{align}
The initial condition is $z_0(t)=1/2,\, z_1(t)=1$. By the complex multiplication theory, we have
\begin{align}
\eta_3(\omega)=\dfrac{\Omega_{A}^{3/2}}{2^{3/2}3^{1/8}\varpi^{3/2}}, ~~C(\omega)=\dfrac{3\Omega_{A}^2}{\varpi^2}.
\end{align}
Therefore we obtain
\begin{align}
\absolute{\partial_{3/2}^{(3N+1)}\eta_3(z)|_{z=\omega}}=\dfrac{\Omega_{A}^{2k-1}}{\varpi^{2k-1}}2^{-3k+5}3^{k-9/4}\absolute{z_{3N+1}(0)}^2.
\end{align}
\end{proof}

\begin{ac}
This paper is based on his master thesis at Kyushu University. The author would like to thank his advisor Shinichi Kobayashi at Kyushu University for suggesting to him the topic in this paper and giving him advice and comments. He would also like to thank S. Yokoyama at Tokyo Metropolitan University for giving him helpful advice for calculators. He is grateful for F. Rodriguez-Villegas and D. Zagier. He has received significant inspiration from their paper. This research did not receive any specific grant from funding agencies in the public, commercial, or not-for-profit sectors.
\end{ac}

%FACULTY OF MATHEMATICS, KYUSHU UNIVERSITY, MOTOOKA 744, NISHI-KU FUKUOKA 819-0395, JAPAN, {\em E-mail address:} \textbf{nomotokeiichiro@gmail.com}

\end{document}